\documentclass[small]{article}
\usepackage{amsmath,amstext,amsthm,amsfonts}
\usepackage{makeidx}
\usepackage{amssymb,latexsym}
\usepackage{txfonts, pxfonts}
\usepackage[english]{babel}
\usepackage[pdftex]{graphicx}
\usepackage{xcolor}
\numberwithin{equation}{section}
\newtheorem{theorem}{\textbf{Theorem}}[section]
\newtheorem{proposition}[theorem]{\textbf{Proposition}}

\newtheorem{lemma}[theorem]{\textbf{Lemma}}

\theoremstyle{definition}

\theoremstyle{remark}

\DeclareMathOperator{\dist}{dist}

\newcommand{\R}{\mathbb{R}}
\newcommand{\N}{\mathbb{N}}

\renewcommand{\(}{\left(}
\renewcommand{\)}{\right)}
\newcommand{\intd}{\ \mathrm{d}}

\begin{document}
	
	\title{A priori estimates for anti-symmetric solutions to a fractional Laplacian equation in a bounded domain}
	\date{}
	
	\author{ \textsc{Chenkai Liu}, \textsc{Shaodong Wang}\footnote{Partially supported by NSFC-12001364.} \textsc{and Ran Zhuo}}

    \maketitle

\begin{abstract}
In this paper, we obtain a priori estimates for the set of anti-symmetric solutions to a fractional Laplacian equation in a bounded domain using a blowing-up and rescaling argument. In order to establish a contradiction to possible blow-ups, we apply a certain variation of the moving planes method in order to prove a monotonicity result for the limit equation after rescaling.
\end{abstract}



\section{Introduction}
The fractional Laplacian in $\mathbb{R}^n$, $n\geqslant 2$, is a non-local pseudo-differential operator of the form:
\begin{equation}\label{eq:fl}
	\begin{aligned}
		(-\Delta)^{\frac{\alpha}{2}} u(x)&=C_{n,\alpha}P.V.\int_{\mathbb{R}^n}\frac{u(x)-u(y)}{|x-y|^{n+\alpha}}\intd y\\
		&=C_{n,\alpha}\lim_{\epsilon\rightarrow0^+}\int_{\mathbb{R}^n\backslash B_{\epsilon}(x)}\frac{u(x)-u(y)}{|x-y|^{n+\alpha}}\intd y,
	\end{aligned}
\end{equation}
where $0<\alpha<2$, $C_{n,\alpha}$ is a positive constant depending only on $\alpha$ and $n$, and $P.V.$ stands for the Cauchy principle value. In order for the integral to make sense, we require $u\in \mathcal{L}^{\alpha}\cap C^{1,1}_{\mathrm{loc}}(\mathbb{R}^n)$ where
\begin{equation*}
	\mathcal{L}^{\alpha}(\mathbb{R}^n)=\{u\in L_{\mathrm{loc}}^1(\mathbb{R}^n) | \int_{\mathbb{R}^n}\frac{|u(x)|}{1+|x|^{n+\alpha}}\intd x<+\infty\}.
\end{equation*}

This non-local operator has been wildly investigated in the past few decades because of its applications in many fields of physical sciences. In particular, it appears in anomalous diffusion, quasi-geostrophic flows, turbulence and water waves,
molecular dynamics, and relativistic quantum mechanics of stars (see for example \cite{bou-geo}, \cite{caf-vas}, \cite{con}, \cite{tar-zas}). It
also has close connections to probability and finance (\cite{app}, \cite{ber}).

In this paper, we study a priori estimates for the set of anti-symmetric solutions to the following fractional Laplacian equations~\eqref{eq:main} in a bounded domain $\Omega$ with $C^2$-boundary:
\begin{align}\label{eq:main}
\begin{cases}
(-\Delta)^{\frac{\alpha}{2}}u(x)=u^p(x), \: u>0 &\text{in}\:\Omega^+,
\\
u(\bar{x},x_n)=-u(\bar{x},-x_n)&\text{in}\: \Omega,\\
u(x)=0 &\text{in}\: \R^n\setminus \Omega,
\end{cases}
\end{align}
where $1<p<\frac{n+\alpha}{n-\alpha}$, $\Omega^+:=\{x\in \Omega : x_n>0\}$ is defined to be the upper half of $\Omega$, $\R^n=\{(\bar{x},x_n), \bar{x}\in \R^{n-1}\}$ is the Euclidean space, and $\Omega$ is  symmetric about the hyperplane $\{x\in \R^n:x_n=0\}$.

A priori estimates play important roles in the study of partial differential equations. In particular, it has proven to be quite useful in establishing the existence of solutions.  Fruitful results as well as useful applications were obtained on the subject. Gidas and Spruck in (\cite{gid-spr}) first established the classical scaling and blowing-up method to derive a priori estimate for some non-linear elliptic equations on a bounded domain in $\mathbb{R}^n$. By adapting the above classical scaling method, Berestycki, Capuzzo-Dolcetta and Nirenberg in  (\cite{ber-dol-nir}) obtained a priori estimate for some elliptic boundary value problem. Moreover, based on the a priori estimates, they proved the existence of solutions in the subcritical case via the Leray-Schauder degree theory. As for the problem with non-local Laplacian operators, Chen, Li and Li in (\cite{che-li-li}) established a direct scaling argument in order to study some fractional Laplacian equations in Euclidean domains. Combining the direct scaling argument with corresponding Liouville theorems, they derived a priori estimates of solutions to certain non-linear fractional Laplace equations. For more historical results on the study of a priori estimates of solutions to non-linear partial differential equations, we refer the reader to \cite{che-li-1}, \cite{che-li-2}, \cite{kri}, \cite{zhu-ly} and the references therein.

In \cite{bar-pez-gar}, Barrios, Del~Pezzo, Garc\'{\i}a-Meli\'{a}n, and
Quaas studied the following non-local elliptic problem on a bounded domain $\tilde{\Omega}$ in $\mathbb{R}^n$:
\begin{align}\label{eq:fl:K}
\begin{cases}
(-\Delta)^{\frac{\alpha}{2}}_K u(x)=u^p(x)+g(x,u), u>0 &\text{in}\:\tilde{\Omega},
\\
u(x)=0 &\text{in}\: \R^n\setminus \tilde{\Omega},
\end{cases}
\end{align}
where
$$(-\Delta)^{\frac{\alpha}{2}}_K u(x)=\frac{C_{n,\alpha}}{2}\int_{\mathbb{R}^n}\frac{2u(x)-u(x+y)-u(x-y)}{|y|^{n+\alpha}} K(y)\intd y,$$
and
$$0<\lambda<K(x)<\Lambda,\,\,\,\,\mbox{in }\tilde{\Omega}, $$
where $\lambda< \Lambda$ are positive constants, $p>1$ and $g$ is a perturbation term which is small in some sense. Note that in the case $K=1$, $(-\Delta)^{\frac{\alpha}{2}}_K$ is reduced to the fractional Laplacian operator~\eqref{eq:fl}. By adapting the
classical scaling method, they obtained the a priori bounds of positive solutions for \eqref{eq:fl:K}. Combining with topological degree, they also derived the existence of solutions. Independently, Chen, Li and Li in \cite{che-li-li} studied the particular case $K=1$ and $g=0$, where  \eqref{eq:fl:K} reduces to the following fractional Laplacian equation:
\begin{align}\label{eq:fl:1}
\begin{cases}
(-\Delta)^{\frac{\alpha}{2}} u(x)=u^p(x), u>0 &\text{in}\:\tilde{\Omega},
\\
u(x)=0 &\text{in}\: \R^n\setminus \tilde{\Omega},
\end{cases}
\end{align}
where $\tilde{\Omega}$ is a bounded domain and $1<p<\frac{n+\alpha}{n-\alpha}$. They carried out the blowing-up and rescaling arguments on the non-local problem directly to obtain a
priori estimates on the solutions of \eqref{eq:fl:1}. 

While there is a large literature dealing with positive solutions to fractional Laplacian equations on domains, very little is known about the anti-symmetric solutions. In this paper, we will apply the direct blowing-up argument similarly as in \cite{che-li-li} to establish a priori estimates on the set of anti-symmetric solutions to equation~\eqref{eq:main}. Our main result is the following theorem:

\begin{theorem}\label{thm: apriori}
	Assume that $u\in \mathcal{L}^\alpha\cap C_{\mathrm{loc}}^{1,1}(\Omega)$ is an anti-symmetric solution to \eqref{eq:main}. Then $$\|u\|_{L^{\infty}(\Omega)}\leqslant C,$$
	for some positive constant $C$ independent of $u$.
\end{theorem}

It is well known that Liouville-type theorems play vital rules in obtain a priori estimates as they could provide contradiction after rescaling and blowing-up. We would like to mention that different types of Liouville theorems for fractional Laplacian operators are already available (see for example \cite{che-fan-yan}, \cite{che-li-li}, \cite{fal-wet}, \cite{qua-xia}, \cite{sil}, \cite{zhu-che-cui-yua}, \cite{zhu-li}).
Due to the anti-symmetric property of our solutions, the blow-up could happen at various locations of $\Omega^+$. Thus we would need Liouville-type theorems in $\mathbb{R}^n$ (\cite{zhu-che-cui-yua}), in $\mathbb{R}^n_+$ (\cite{che-fan-yan,zhu-li}) and in quarter spaces of $\mathbb{R}^n$ in order to derive a contradiction. As far as we know there is no Liouville theorems for the following equation~\eqref{eq: limit} in the quarter space situations. 
\begin{align}\label{eq: limit}
	\begin{cases}
		(-\Delta)^{\frac{\alpha}{2}}v(x)=v^p(x), u>0  &\text{in}\:\{x_1>0, x_2>0\},
		\\
		v \: \text{is anti-symmetric about} \: &\{x_1=0\},
		\\
		v(x)=0 \: &\text{in} \: \{x_2\leqslant 0 \},
	\end{cases}
\end{align}
where $1<p<\frac{n+\alpha}{n-\alpha}$. 
In order to overcome the difficulty, we apply the moving planes method for fractional Laplacian equations to obtain the following monotonicity result for bounded solutions of ~\eqref{eq: limit}:

\begin{theorem}\label{prop: monotonicity}
	Assume $v\in L^{\infty}\cap C^{1,1}_{\mathrm{loc}}(\R^n)$ is a solution of \eqref{eq: limit} for $1<p<\frac{n+\alpha}{n-\alpha}$, then $v(x)$ is monotone increasing in the $x_2$ direction in $\{x_1>0, x_2>0\}$.
\end{theorem}

This is one of the novelty of our present work. Another novelty is that we obtain various maximum principle results in the case of anti-symmetric functions. We believe the ideas contained in these results could have potential applications in the study of other types of fractional Laplacian problems.

This paper is organized as follows. In Section 2 we present various types of maximum principles in the case of anti-symmetric functions to be used in later sections. The important monotonicity result  Theorem~\ref{prop: monotonicity} is proved in Section 3 using the method of moving planes. Finally in Section 4, we give the proof of our main result Theorem~\ref{thm: apriori} using the direct rescaling and blow-up argument.

\section{Maximum principles}
Let us start by fixing some notations in this section. A point in $\R^n$ will be denoted by :
	\begin{equation*}
		x=(x_1,x_2,x^\prime),\: x^\prime\in\mathbb{R}^{n-2}.
	\end{equation*}
	Let $T_i$ ($i=1,2$) denote the reflection of any point $x\in \R^n$ about the hyperplane $\{x_i=0\}$ , i.e.,
	\begin{equation*}
		T_1(x_1,x_2,x^\prime)=(-x_1,x_2,x^\prime),\: T_2(x_1,x_2,x^\prime)=(x_1,-x_2,x^\prime).
	\end{equation*}
	Denote by $I$, $A$, $B$ to be the areas in $\R^n$ respectively:
	\begin{equation*}
		\begin{aligned}
			I=(0,+\infty)\times(0,+\infty)\times\mathbb{R}^{n-2},\\
			A=(0,+\infty)\times (0,1)\times\mathbb{R}^{n-2},\\
			B=(0,+\infty)\times (1,+\infty)\times\mathbb{R}^{n-2}.
		\end{aligned}
	\end{equation*}
We are going to prove some maximum principles for anti-symmetric solutions of certain fractional Laplacian equations to be used in the following sections. Let us first prove the following decay estimates:
\begin{proposition}\label{prop: decay}
	Let
	$w\in C^0(\mathbb{R}^n)\cap L^\infty(\mathbb{R}^n)\cap C^{1,1}_{\mathrm{loc}}\left(A\right)$ be a solution to the following equation:
	\begin{equation*}
		\begin{cases}
			w(T_1x)=w(T_2x)=-w(x)\:&\text{in} \: \mathbb{R}^n,\\
			w\geqslant 0\:&\text{in}\: B.
		\end{cases}
	\end{equation*}
	Suppose that
	\begin{equation}\label{eq: mini}
		\min_{x\in A}w(x)=w(x^*)<0,
	\end{equation}
for some $x^*\in A$.
	Then
	\begin{equation}\label{eq: decay}
		(-\Delta)^{\frac{\alpha}{2}} w(x^*)\leqslant C_{0}\dist (x^*,\partial A)^{-\alpha}w(x^*),
	\end{equation}
	where $C_0>0$ is a constant depending only on $n$ and $\alpha$.
\end{proposition}
\begin{proof}
	Straightforward calculation using the definition of the fractional Laplacian operator~\eqref{eq:fl} produces:
	\begin{equation}
		\label{eq: frac mini}
		\begin{aligned}
			(-\Delta)^{\frac{\alpha}{2}}w(x^*)&=C_{n,\alpha}P.V. \int_{\mathbb{R}^n}\frac{w(x^*)-w(y)}{|x^*-y|^{n+\alpha}}\intd y\\
			&=C_{n,\alpha}\bigg[P.V. \int_{A\cup T_1A\cup T_2A\cup T_1T_2A}\frac{w(x^*)-w(y)}{|x^*-y|^{n+\alpha}}\intd y \\
			&\qquad\qquad+\int_{B\cup T_1B\cup T_2B\cup T_1T_2B}\frac{w(x^*)-w(y)}{|x^*-y|^{n+\alpha}}\intd y\bigg]\\
			&=C_{n,\alpha}\bigg[P.V. \int_{A}\underbrace{(w(x^*)-w(y))}_{\leqslant 0\ \text{by} \ \eqref{eq: mini}}K(x^*,y)\intd y \\
			&\qquad+2 w(x^*)\int_{A}\frac{1}{|x^*-T_1y|^{n+\alpha}}+\frac{1}{|x^*-T_2y|^{n+\alpha}}\intd y \\
			&\qquad+\int_{B\cup T_1B\cup T_2B\cup T_1T_2B}\frac{w(x^*)}{|x^*-y|^{n+\alpha}}\intd y-\int_{B}\underbrace{w(y)}_{\geqslant 0}K(x^*,y)\intd y\bigg]\\
			&\leqslant L(x^*)w(x^*),
		\end{aligned}
	\end{equation}
	where $K(x,y)$ and $L(x)$ are defined as follow:
	\begin{align*} &K(x,y) =\frac{1}{|x-y|^{n+\alpha}}-\frac{1}{|x-T_1y|^{n+\alpha}} -\frac{1}{|x-T_2y|^{n+\alpha}}+\frac{1}{|x-T_1T_2y|^{n+\alpha}}, \\
		\label{eq:min4}&L(x)=C_{n,\alpha}\left[2\int_{A}\frac{1}{|x-T_1y|^{n+\alpha}}+\frac{1}{|x-T_2y|^{n+\alpha}}\intd y+\int_{B\cup T_1B\cup T_2B\cup T_1T_2B}\frac{1}{|x-y|^{n+\alpha}}\intd y\right].
	\end{align*}
	We have used in \eqref{eq: frac mini} the fact that $K(x,y)\geqslant 0$ (see Lemma 3.2 of ~\cite{li-zhuo}) for $x\in A$ and $y\in A\cup B$.
	It follows from direct calculation (see for example Lemma 2.1 of ~\cite{che-hua-li}) that
	\begin{equation*}
		L(x)\leqslant C_{0}\dist (x,\partial A)^{-\alpha},
	\end{equation*}
	for $x\in A$ and some constant $C_0>0$ depending only on $n$ and $\alpha$.
	Now we have derived \eqref{eq: decay} and the proof is finished.
\end{proof}

The next proposition is a maximum principle for a fractional Laplacian equation in a special form:

\begin{proposition}[Maximum principle]\label{prop: mp}
	Let
	$w\in C^0(\mathbb{R}^n)\cap L^\infty(\mathbb{R}^n)\cap C^{1,1}_{\mathrm{loc}}\left(A\right)$ satisfies the following:
	\begin{equation}
		\label{eq: mp}
		\begin{cases}
			(-\Delta)^{\frac{\alpha}{2}}w+cw\geqslant0\:&\text{in}\: A,\\
			w(T_1x)=w(T_2x)=-w(x)\:&\text{for any}\: x\in\mathbb{R}^n,\\
			w\geqslant 0\:&\text{in}\: B,\\
		    w(x)\rightarrow 0 \: \text{as} \: x\rightarrow \infty.&
		\end{cases}
	\end{equation}
	Then there exists $M>0$ depending only on $n$ and $\alpha$, such that whenever $c\geqslant -M$, we have
	\begin{equation*}
		w\geqslant 0\qquad\text{in}\quad A.
	\end{equation*}
\end{proposition}
\begin{proof}
	We will prove this proposition by contradiction. Since $w\in C^0(\R^n)$ and decays at infinity, we can suppose that
	\begin{equation*}
		\min_{x\in A}w(x)=w(x^*)<0,
	\end{equation*}
for some $x^*\in A$.
	Then we know from \eqref{eq: decay} that
	\begin{equation*}
		(-\Delta)^{\frac{\alpha}{2}} w(x^*)\leqslant C_{0}\dist (x^*,\partial A)^{-\alpha}w(x^*).
	\end{equation*}
	Note that $\dist (x^*,\partial A)\leqslant \frac{1}{2}$. We can choose $M<4^{\frac{\alpha}{2}}C_0$ to derive
	\begin{equation*}
		(-\Delta)^{\frac{\alpha}{2}} w(x^*)+cw(x^*)\leqslant[4^{\frac{\alpha}{2}}C_0-M]w(x^*)<0.
	\end{equation*}
	 This contradicts the equation \eqref{eq: mp} thus proves that
	\begin{equation*}
		w\geqslant 0\qquad\text{in}\quad A.
	\end{equation*}
	This ends the proof of Proposition~\ref{prop: mp}.
\end{proof}

In the following we derive a version of the strong maximum principle to be used in the next section.

\begin{proposition}[Strong maximum principle]\label{prop: smp}
	Let
	$w\in C^0(\mathbb{R}^n)\cap L^\infty(\mathbb{R}^n)\cap C^{1,1}_{\mathrm{loc}}\left(A\right)$ satisfies the following:
	\begin{equation*}
		\begin{cases}
			(-\Delta)^{\frac{\alpha}{2}}w+cw\geqslant0\:&\text{in}\: A,\\
			w(T_1x)=w(T_2x)=-w(x)\:&\text{for any}\: x\in\mathbb{R}^n,\\
			w\geqslant 0\:&\text{in}\: B, \\
			w(x)\rightarrow 0 \: \text{as} \: x\rightarrow \infty,&
		\end{cases}
	\end{equation*}
	where $c$ is a bounded function. Then either
	\begin{equation*}
		w> 0\qquad\text{in}\quad A,
	\end{equation*} or
	\begin{equation*}
		w\equiv 0\qquad\text{in}\quad A.
	\end{equation*}
\end{proposition}

\begin{proof}
	Suppose that $w$ is not constant zero but $w(x^*)=0$ for some $x^*\in A$. We calculate as in \eqref{eq: frac mini} to obtain that
	\begin{equation*}
		\begin{aligned}
			0&\leqslant(-\Delta)^{\frac{\alpha}{2}}w(x^*)+c(x^*)w(x^*)\\&=C_{n,\alpha}P.V. \int_{\mathbb{R}^n}\frac{w(x^*)-w(y)}{|x^*-y|^{n+\alpha}}\intd y\\
			&=C_{n,\alpha}\bigg[P.V. \int_{A}-w(y)K(x^*,y)\intd y \\
			&-\int_{B}w(y)K(x^*,y)\intd y\bigg]\\
			&\leqslant 0.
		\end{aligned}
	\end{equation*}
	This implies that $w(y)\equiv 0$ in $A$ which is a contradiction to our assumption. Here in the last inequality we use the maximum principle Proposition~\ref{prop: mp}. This ends the proof of Proposition~\ref{prop: smp}.
\end{proof}

\section{A monotonicity result for a limit equation}

Let us consider the equation \eqref{eq: limit}
which is a special case of our main equation ~\eqref{eq:main} with $1<p<\frac{n+\alpha}{n-\alpha}$. We will prove Theorem~\ref{prop: monotonicity} in this section. We first present a useful lemma on the regularity and decay property of any solution to equation~\eqref{eq: limit}.

\begin{lemma}\label{lemma: decay}
	Suppose $v\in L^{\infty}\cap C^{1,1}_{\mathrm{loc}}(\mathbb{R}^n)$ is a solution to \eqref{eq: limit}, then $v\in C^{\frac{\alpha}{2}}(\mathbb{R}^n)$ and $v(x)\rightarrow 0$ as $x\rightarrow \infty$.
\end{lemma}
\begin{proof}
	See \cite{li-wu} or \cite{sil} for the interior regularity results. For regularity up to the boundary, one could adapt the argument in the proof of Theorem 1 in \cite{che-li-li}. Now let us prove the decay property by contradiction argument. Suppose there exists a sequence $\{x^m\}\in\mathbb{R}^n$ such that $|v(x^m)|\geqslant c_0>0$ for any $m$ and that $\lim\limits_{m\rightarrow \infty}|x^m|=+\infty$.
	
	By the anti-symmetry property of $v$, we may without loss of generality assume:
	\begin{equation*}
		x^m\in I =(0,+\infty)\times(0,+\infty)\times\mathbb{R}^{n-2}.
	\end{equation*}
	Define $v_m(x)=v(x-x^m)$. By the Arzel\`a-Ascoli Theorem, there exist a subsequence of $\{v_m\}$, still denoted by $\{v_m\}$ for convenience, that converges to some non-trivial $v_*\geqslant 0$ in $C^{\beta}(\mathbb{R}^n)$ for any $0<\beta<\frac{\alpha}{2}$.
	
	We are going to divide our argument into the following cases:
	\begin{enumerate}
		\item [Case 1:] If $x^m_1\rightarrow x^*_1<+\infty$ as $m\rightarrow \infty$, then $v_*$ satisfies:
		\begin{equation*}
			\begin{cases}
				(-\Delta)^{\frac{\alpha}{2}} v_*=v_*^p\: &\text{in}\: (-x^*_1,+\infty)\times\mathbb{R}^{n-1},\\
				v_*\text{ is anti-symmetric about } &\{x_1=-x^*_1\}.
			\end{cases}
		\end{equation*}
		This contradicts the Liouville theorem of anti-symmetric functions (see \cite{zhu-li} for a reference).
		\item [Case 2:] If $x^m_2\rightarrow x^*_2<+\infty$ as $m\rightarrow \infty$, then  $v_*$ satisfies:
		\begin{equation*}
			\begin{cases}
				(-\Delta)^{\frac{\alpha}{2}} v_*=v_*^p\:&\text{in}\: \mathbb{R}\times(-x^*_2,+\infty)\times\mathbb{R}^{n-2},\\
				v_*=0&\text{on} \: \mathbb{R}\times(-\infty,-x^*_2]\times\mathbb{R}^{n-2}.
			\end{cases}
		\end{equation*}
		This contradicts the Liouville theorem in the case of half space (See \cite{che-fan-yan} for a reference).
		\item [Case 3:]If both $x^m_i\rightarrow +\infty$ $(i=1,2)$ as $m\rightarrow +\infty$, then  $v_*$ satisfies:
		\begin{equation*}
			(-\Delta)^{\frac{\alpha}{2}} v_*=v_*^p\:\text{in}\:\mathbb{R}^n.
		\end{equation*}
		This contradicts the Liouville theorem in the whole space $\mathbb{R}^n$ (See \cite{zhu-che-cui-yua} for a reference).
	\end{enumerate}
	These three contradictions show that $v(x)\rightarrow 0$ as $|x|\rightarrow+\infty$ and ends the proof.
\end{proof}

Now let us prove the monotonicity property that we need in the proof of our main theorem. We are going to prove Theorem~\ref{prop: monotonicity} via the method of moving planes:

\begin{proof}
	
	In the following, we always denote
	$$\R^n:=\{x=(x_1,x_2,x'), x_1, x_2\in \R,  x^\prime\in\mathbb{R}^{n-2}\}.$$
	For any $\lambda>0$, let us denote
	$$ \Gamma_\lambda:=\{(x_1,\lambda,x'): x_1\in \R\},$$
	$$\Sigma_\lambda:=\{(x_1, x_2, x'): x_1\in \R, 0<x_2<\lambda\},$$ and
	$$\Sigma^+_\lambda:=\{(x_1, x_2, x'): x_1>0, 0<x_2<\lambda\}.$$
	Moreover, let $v_\lambda(x):=v(x_1,2\lambda-x_2,x')$ be the reflection of $v(x)$ about the hyperplane $\Gamma_\lambda$. In order to compare $v(x)$ and $v_\lambda(x)$, let $w_\lambda(x):=v_\lambda(x)-v(x)$. We are going to prove that $w_\lambda(x)>0$ in $\Sigma^+_\lambda$ for any $\lambda>0$. It follows from a direct calculation that $w_\lambda(x)$ satisfies the following equation:
	\begin{align*}
		\begin{cases}
			(-\Delta)^{\frac{\alpha}{2}}w_\lambda(x)=v^p_\lambda(x)-v^p(x)=p\varphi^{p-1}w_\lambda(x) \: & \text{in} \: \Sigma_\lambda^+,\\
			w_\lambda(x_1,2\lambda-x_2, x')=-w_\lambda(x_1, x_2, x') \\
			w_\lambda(x_1,x_2, x')=-w_\lambda(-x_1, x_2, x') \\
			w_\lambda(x)\geqslant 0 \: & \text{in} \: \{x_2 \leqslant0\},
		\end{cases}
	\end{align*}
	where $\varphi$ is a bounded function between $v(x)$ and $v_\lambda(x)$. We start by taking $\lambda>0$ small enough. In this case, we apply the rescaling $u_\lambda(x)=w_\lambda(\lambda x)$ such that $u_\lambda$ satisfies:
	\begin{align}\label{eq: v}
		\begin{cases}
			(-\Delta)^{\frac{\alpha}{2}}u_\lambda(x)=\lambda^\alpha c(x)u_\lambda(x) \: & \text{in} \: \Sigma_1^+,\\
			u_\lambda(x_1,2-x_2, x')=-u_\lambda(x_1, x_2, x') \\
			u_\lambda(x_1,x_2, x')=-u_\lambda(-x_1, x_2, x') \\
			u_\lambda(x)\geqslant 0 \: & \text{in} \: \{x_2 \leqslant0\},
		\end{cases}
	\end{align}
	where $c(x)$ is a bounded function.
	By taking $\lambda>0$ small enough, it follows from maximum principle for narrow domains Proposition~\ref{prop: mp} that $u_\lambda(x)\geqslant 0$ in $\Sigma_1^+$. Therefore $w_\lambda(x)\geqslant 0$ in $\Sigma_\lambda^+$. The equality cannot always hold since otherwise we obtain that $v(x)\equiv0$ in $\R^n$ which contradicts the assumption of equation~\eqref{eq: limit}. Thus we can apply the strong maximum principle Proposition~\ref{prop: smp} to obtain that $w_\lambda>0$ in $\Sigma_\lambda^+$. This gives a starting point for our moving.
	
	Now let
	$$ \lambda_0:=sup\{\lambda: w_\lambda>0 \: \text{in} \: \Sigma_\mu^+ \: \text{for any} \: 0<\mu<\lambda\}.$$
	Again from strong maximum principle Proposition~\ref{prop: smp},
	$w_{\lambda_0}>0$ in  $\Sigma_{\lambda_0}^+.$ We are going to show that for some $\varepsilon>0$,
	$$w_\lambda>0 \: \text{in} \: \Sigma_\lambda^+ \: \text{for} \: \lambda\in [\lambda_0, \lambda_0+\varepsilon).$$ We begin by proving the following claim:
	
	\textbf{Claim}: For any given $\delta>0$ small, there exists $\varepsilon>0$ small such that $w_{\lambda}\geqslant 0$ in $\Sigma^+_{\lambda_0-\delta}$ for any $\lambda\in [\lambda_0,\lambda_0+\varepsilon)$.
	
	Suppose the claim is false, there exists some $\delta>0$, a sequence $\lambda_n>\lambda_0$, $\lambda_n\rightarrow\lambda_0$ such that $$\inf_{\Sigma^+_{\lambda_0-\delta}}w_{\lambda_n}<0.$$ We may assume that there exists $\{x^n\}\in \Sigma^+_{\lambda_0-\delta}$ such that $$\inf_{\Sigma^+_{\lambda_0-\delta}}w_{\lambda_n}=w_{\lambda_n}(x^n)<0.$$
	If up to a subsequence $x^n\rightarrow \bar{x}\in \Sigma^+_{\lambda_0}$, we have that
	$$0<w_{\lambda_0}(\bar{x})=\lim_{n\rightarrow \infty}w_{\lambda_n}(x^n)\leqslant 0,$$ which is impossible. Therefore we have either $d(x^n, \partial \Sigma^+_{\lambda_n})\rightarrow 0$ or $x^n\rightarrow \infty$.
	If $x^n\rightarrow \infty$, then using the equation satisfied by $w_{\lambda_n}$ and the decay estimates Proposition~\ref{prop: decay} we obtain $$(-\Delta )^{\frac{\alpha}{2}}w_{\lambda_n}(x^n)=p\varphi^{p-1}(x^n)w_{\lambda_n}(x^n)\leqslant C_0d(x^n, \partial \Sigma^+_{\lambda_n})^{-\alpha}w_{\lambda_n}(x^n),$$
	where $\varphi(x)$ is between $v(x)$ and $v_{\lambda_n}(x)$.
	 Due to the decay property Lemma~\ref{lemma: decay} we have that $\varphi(x^n)\rightarrow 0$ as $n\rightarrow \infty$. Since $d(x^n, \partial \Sigma^+_{\lambda_n})\leqslant \frac{\lambda_n}{2}$ and $w_{\lambda_n}(x^n)<0$, we have 
	 $$p\varphi^{p-1}(x^n)\geqslant C_0(\frac{\lambda_n}{2})^{-\alpha}.$$ 
	 Letting $n\rightarrow \infty$ we obtain
	 $$0\geqslant C_0(\frac{\lambda_0}{2})^{-\alpha}>0$$
	 which is a contradiction! Thus the only possibility is that $d(x^n, \partial \Sigma^+_{\lambda_n})\rightarrow 0$. Now once again by Proposition~\ref{prop: decay} and the equation satisfied by $w_{\lambda_n}$ we obtain
	$$ 0=(-\Delta )^{\frac{\alpha}{2}}w_{\lambda_n}(x^n)+c(x^n)w_{\lambda_n}(x^n)\leqslant(C_0d(x^n, \partial \Sigma^+_{\lambda_n})^{-\alpha}+c(x^n))w_{\lambda_n}(x^n)<0,$$ where $n$ is chosen large enough such that $C_0d(x^n, \partial \Sigma^+_{\lambda_n})^{-\alpha}+c(x^n)>0$, and $c(x)$ is a bounded function. This is again a contradiction. Thus we have proved our claim.
	
	Now in the region $\Sigma^+_{\lambda}\backslash \Sigma^+_{\lambda_0-\delta}$, $w_\lambda$ satisfies:
	\begin{align}\label{eq: w}
		\begin{cases}
			(-\Delta)^{\frac{\alpha}{2}}w_\lambda(x)=c(x)w_\lambda(x)\: & \text{in} \: \Sigma^+_{\lambda}\backslash \Sigma^+_{\lambda_0-\delta},\\
			w_\lambda(x_1,2\lambda-x_2, x')=-w_\lambda(x_1, x_2, x') \\
			w_\lambda(x_1,x_2, x')=-w_\lambda(-x_1, x_2, x') \\
			w_\lambda(x)\geqslant 0 \: & \text{in} \: \{x_2 \leqslant\lambda_0-\delta\}
		\end{cases}
	\end{align}
	where $c(x)$ is bounded.
	Since $\delta+\varepsilon$ are chosen small, we could apply maximum principle for narrow domains Proposition~\ref{prop: mp} again to derive that $w_\lambda(x)\geqslant 0$ in $\Sigma_\lambda^+\backslash \Sigma^+_{\lambda_0-\delta}$. Combine with the previous claim and apply strong maximum principle again we derive that $w_\lambda(x)>0$ in $\Sigma_\lambda^+$ which gives a contradiction to the definition of $\lambda_0$. Thus we have proved that $w_\lambda(x)>0$ in $\Sigma_\lambda^+$  for any $\lambda>0$. The conclusion of Theorem~\ref{prop: monotonicity} follows immediately.
\end{proof}

\section{Proof of Theorem~\ref{thm: apriori}}
Now we are ready to prove our main result. The proof is based on a direct rescaling and blowing-up argument. As we mentioned in the Introduction, various types of Liouville results are needed. In the special case of the quarter space, we will use the monotonicity result Theorem~\ref{prop: monotonicity} to obtain a contradiction.
\begin{proof}
	We will prove Theorem~\ref{thm: apriori} by contradiction. Assume the conclusion is false, then we may assume that there exists a sequence of solutions $\{u_k\}_{k\in \N}$ to \eqref{eq:main} and points $\{x^k\}_{k\in \N}\in \Omega^+$ such that
	$$u_k(x^k):=\max_\Omega u_k=m_k\rightarrow +\infty,$$
	as $k\rightarrow +\infty$. Let $\lambda_k:=m_k^{\frac{1-p}{\alpha}}$.
	Define $v_k(x)=\frac{1}{m_k}u_k(\lambda_kx+x^k)$. Then we have
	$$(-\Delta)^{\frac{\alpha}{2}}v_k(x)=v_k^p(x), \: x\in \Omega_k^+,$$
	where $\Omega_k^+:=\left\lbrace x\in \R^n: x=\frac{y-x^k}{\lambda_k}, y\in \Omega^+\right\rbrace$. We define $\partial^+\Omega^+:=\left\lbrace x\in \partial \Omega^+: x_n>0\right\rbrace $ and $\partial' \Omega^+:=\left\lbrace x\in \partial \Omega^+: x_n=0\right\rbrace.$ Let $d_k:=\dist(x^k, \partial\Omega^+), d'_k:=\dist(x^k, \partial'\Omega^+)$ and $d^+_k:=\dist(x^k, \partial^+\Omega^+)$. We will divide our argument into several cases:
	
	\begin{enumerate}
		\item [Case 1:] $\lim\limits_{k\rightarrow +\infty}\frac{d_k}{\lambda_k}\rightarrow +\infty.$
		
		In this case, $\Omega_k^+\rightarrow \R^n$ as $k\rightarrow +\infty$. Moreover, $v_k(0)=1$ and $\|v_k\|_{L^{\infty}}\leqslant 1$. In view of Schauder estimates for fractional Laplacian equations (see for example Proposition 2.7, 2.8 and 2.9 of \cite{sil}), we have that there exists a function $v$ such that
		$$v_k(x)\rightarrow v(x)$$ and
		\begin{equation*}
		(-\Delta)^{\frac{\alpha}{2}}v_k(x)\rightarrow(-\Delta)^{\frac{\alpha}{2}}v(x),
		\end{equation*}
		as $k\rightarrow +\infty$.  Therefore
		\begin{equation}\label{eq: limit 1}
		(-\Delta)^{\frac{\alpha}{2}}v(x)=v^p(x) \: \text{in} \: \R^n.
		\end{equation}
		It follows from classical Liouville theorem (See \cite{zhu-che-cui-yua}) that \eqref{eq: limit 1} has no positive solution. However, $$v(0)=\lim\limits_{k\rightarrow +\infty} v_k(0)=1.$$
		This is a contradiction.
		
		\item [Case 2:] $\lim\limits_{k\rightarrow +\infty}\frac{d'_k}{\lambda_k}\rightarrow +\infty$, $\lim\limits_{k\rightarrow +\infty}\frac{d^+_k}{\lambda_k}\rightarrow C>0.$
		
	    In this case, $\Omega_k^+\rightarrow \left\lbrace x\in \R^n: x_n>-C\right\rbrace $ as $k\rightarrow +\infty$. Moreover, $v_k(0)=1$ and $\|v_k\|_{L^{\infty}}\leqslant 1$. In view of Schauder estimates for fractional Laplacian equations, there exists a function $v$ such that
	    $$v_k(x)\rightarrow v(x)$$ and
	    \begin{equation*}
	    (-\Delta)^{\frac{\alpha}{2}}v_k(x)\rightarrow(-\Delta)^{\frac{\alpha}{2}}v(x),
	    \end{equation*}
	    as $k\rightarrow +\infty$.  Therefore
	    \begin{equation}\label{eq: limit 2}
	    \begin{cases}
	    (-\Delta)^{\frac{\alpha}{2}}v(x)=v^p(x) \: &\text{in} \: \{ x_n>-C\}, \\
	    v(x)=0 \: & \text{in} \: \{ x_n\leqslant -C\}.
	    \end{cases}
	    \end{equation}
	    It follows from Liouville theorem (See \cite{che-fan-yan}) that \eqref{eq: limit 2} has no positive solution. However, $$v(0)=\lim\limits_{k\rightarrow +\infty} v_k(0)=1.$$
	    This is a contradiction.
	
	    \item [Case 3:] $\lim\limits_{k\rightarrow +\infty}\frac{d'_k}{\lambda_k}\rightarrow C>0$, $\lim\limits_{k\rightarrow +\infty}\frac{d^+_k}{\lambda_k}\rightarrow +\infty.$
	
	    In this case, $\Omega_k^+\rightarrow \left\lbrace x\in \R^n: x_1>-C\right\rbrace $ as $k\rightarrow +\infty$. Moreover, $v_k(0)=1$ and $\|v_k\|_{L^{\infty}}\leqslant 1$. In view of Schauder estimates for fractional Laplacian equations, there exists a function $v$ such that
	    $$v_k(x)\rightarrow v(x)$$ and
	    \begin{equation*}
	    (-\Delta)^{\frac{\alpha}{2}}v_k(x)\rightarrow(-\Delta)^{\frac{\alpha}{2}}v(x),
	    \end{equation*}
	    as $k\rightarrow +\infty$.  Therefore
	    \begin{equation}\label{eq: limit 4}
	    \begin{cases}
	    (-\Delta)^{\frac{\alpha}{2}}v(x)=v^p(x)\: &\text{in} \: \{ x_1>-C\}, \\
	    v \: \text{is anti-symmetric about} \: &\{x_1=-C\}.
	    \end{cases}
	    \end{equation}
	    It follows from Liouville theorem for anti-symmetric equations (See \cite{zhu-li}) that \eqref{eq: limit 4} has no positive solution. However, $$v(0)=\lim\limits_{k\rightarrow +\infty} v_k(0)=1.$$
	    This is a contradiction.
	
	    \item [Case 4:] $\lim\limits_{k\rightarrow +\infty}\frac{d'_k}{\lambda_k}\rightarrow C_1>0$, $\lim\limits_{k\rightarrow +\infty}\frac{d^+_k}{\lambda_k}\rightarrow C_2>0.$
	
	   In this case, $\Omega_k^+\rightarrow \left\lbrace x\in \R^n: x_1>-C_1, x_n>-C_2\right\rbrace $ as $k\rightarrow +\infty$. Moreover, $v_k(0)=1$ and $\|v_k\|_{L^{\infty}}\leqslant 1$. In view of Schauder estimates for fractional Laplacian equations, there exists a function $v$ such that
	   $$v_k(x)\rightarrow v(x)$$ and
	   \begin{equation*}
	   (-\Delta)^{\frac{\alpha}{2}}v_k(x)\rightarrow(-\Delta)^{\frac{\alpha}{2}}v(x),
	   \end{equation*}
	   as $k\rightarrow +\infty$.  Therefore
	    \begin{align*}
	    \begin{cases}
	    (-\Delta)^{\frac{\alpha}{2}}v(x)=v^p(x) &\text{in}\:\{x_1>-C_1, x_n>-C_2\},
	    \\
	    v \: \text{is anti-symmetric about} \: &\{x_1=-C_1\},
	    \\
	    v(x)=0 \: &\text{in} \: \{x_n\leqslant -C_2\}.
	    \end{cases}
	    \end{align*}
    	We could then apply the monotonicity Theorem~\ref{prop: monotonicity} to obtain that $v(x)$ is monotone increasing in the $x_n$ direction in $\{x_1>-C_1, x_n>-C_2\}$. However, $v(x)$ has a maximum point at $0$  with $v(0)=1$. Thus we obtain a contradiction.
	
	    \item [Case 5:] $\lim\limits_{k\rightarrow +\infty}\frac{d'_k}{\lambda_k}\rightarrow +\infty$, $\lim\limits_{k\rightarrow +\infty}\frac{d^+_k}{\lambda_k}\rightarrow 0.$
	
		In this case, we follow the argument as in the proof of Theorem 1 in ~\cite{che-li-li} to obtain a contradiction. For the reader's convenience, we include the details here. Now there exists a point $x^0\in \partial^+\Omega$ such that, up to a subsequence, $x^k\rightarrow x^0$ as $k\rightarrow +\infty$. Let $p^k:=\frac{x^0-x^k}{\lambda_k}$. Then we have $$p^k\rightarrow 0$$ as $k\rightarrow +\infty$ and $$v_k(p^k)=0.$$ We are going to show that $v_k$ is uniformly H\"older continuous near $p^k$, i.e.
		\begin{equation}\label{eq: holder}
		|v_k(x)-v_k(p^k)|\leqslant C|x-p^k|^{\frac{\alpha}{2}},
		\end{equation}
		for some constant $C>0$. Once we have \eqref{eq: holder}, we would have
		$$1=v_k(0)-v_k(p^k)\rightarrow 0$$ as $k\rightarrow +\infty$, which gives a contradiction. Now we only have to prove \eqref{eq: holder}. In order to do this, we need to construct some auxiliary function. Since $\partial \Omega$ is $C^2$, for $k$ large enough, there is a unit ball contained in $\Omega_k^c$ that is tangent to $\partial \Omega_k$ at $p^k$. Without losing generality we assume the unit ball is centred at the origin. Let $\psi_1(x)=C(1-|x|^2)^{\frac{\alpha}{2}}$ for $x\in B_1(0)$, and $\psi_1(x)=0$ for $x\in B_1^c(0)$ where $C>0$ is a normalization constant such that
		$$(-\Delta)^{\frac{\alpha}{2}}\psi_1(x)=1.$$
		Let $$\psi_2(x)=\frac{1}{|x|^{n-\alpha}}\psi_1\(\frac{x}{|x|^2}\)$$
		be the Kelvin transformation of $\psi_1(x)$. Then
		$$(-\Delta)^{\frac{\alpha}{2}}\psi_2(x)=\frac{1}{|x|^{n+\alpha}}, \: x\in B_1^c(0).$$
		Moreover, we choose $\xi(x)$ to be a smooth cut-off function such that $0\leqslant \xi(x) \leqslant 1$ in $\R^n$, $\xi(x)=0$ in $B_1(0)$ and $\xi(x)=1$ in $B_3^c(0)$. It is easy to check that
		$$(-\Delta)^{\frac{\alpha}{2}}\xi_1(x)\geqslant -C$$
		for some constant $C>0$. Now we let
		$$\varphi(x)=t\psi_2(x)+\xi(x), \: t>0.$$
		In the region $D:=(B_3(0)\backslash B_1(0))\cap \Omega_k$ we have for $t$ large enough
		$$(-\Delta)^{\frac{\alpha}{2}}\varphi_1(x)\geqslant \frac{t}{|x|^{n+\alpha}}-C\geqslant 1.$$
		Therefore we have
		\begin{align*}
		\begin{cases}
		(-\Delta)^{\frac{\alpha}{2}}(\varphi-v_k)\geqslant 0 \: &x\in D, \\
		\varphi-v_k>0 \: & x\in D^c.
		\end{cases}
		\end{align*}
		From maximum principle (see \cite{che-li-li2}) we see that
		$$\varphi\geqslant v_k, \: x\in D.$$
		Since $$\psi_2(x)-\psi_2(p^k)=\frac{1}{|x|^n}(|x|-1)^{\frac{\alpha}{2}}(|x|+1)^{\frac{\alpha}{2}}\leqslant C(|x|-1)^{\frac{\alpha}{2}},$$
		for $x\in D$ and $\xi(x)$ is smooth, we see that $\varphi(x)$ is H\"older continuous in $D$. Hence
		$$0\leqslant v_k(x)-v_k(p^k)=v_k(x)\leqslant \varphi(x)\leqslant \varphi(x)-\varphi(p^k)\leqslant C|x-p^k|^{\frac{\alpha}{2}}.$$
		This proves \eqref{eq: holder} and ends the proof of Case 4.

	    \item [Case 6:] $\lim\limits_{k\rightarrow +\infty}\frac{d'_k}{\lambda_k}\rightarrow C>0$, $\lim\limits_{k\rightarrow +\infty}\frac{d^+_k}{\lambda_k}\rightarrow 0.$
	
	    \item [Case 7:] $\lim\limits_{k\rightarrow +\infty}\frac{d'_k}{\lambda_k}\rightarrow 0$, $\lim\limits_{k\rightarrow +\infty}\frac{d^+_k}{\lambda_k}\rightarrow +\infty.$
	
	    \item [Case 8:] $\lim\limits_{k\rightarrow +\infty}\frac{d'_k}{\lambda_k}\rightarrow 0$, $\lim\limits_{k\rightarrow +\infty}\frac{d^+_k}{\lambda_k}\rightarrow C>0.$
	
	    \item [Case 9:] $\lim\limits_{k\rightarrow +\infty}\frac{d'_k}{\lambda_k}\rightarrow 0$, $\lim\limits_{k\rightarrow +\infty}\frac{d^+_k}{\lambda_k}\rightarrow 0.$

	    We could proceed as in Case 5 to obtain contradictions for Cases 6-9 with some obvious modifications. This ends the proof of Theorem~\ref{thm: apriori}.
	\end{enumerate}

\end{proof}

\bigskip\noindent
{\bf{Acknowledgement.}} The authors would like to thank Prof. Congming Li for his fruitful discussions on this subject.

\bigskip\noindent

\textsc{Chenkai Liu\\
	School of Mathematical Sciences,\\
	Shanghai Jiao Tong University\\
	Shanghai, 200240, People’s Republic of China} \\
e-mail: {\bf{Lck0427@sjtu.edu.cn}}\\

\textsc{Shaodong Wang\\
	School of Mathematics and Statistics,\\
	Nanjing University of Science and Technology\\
	Nanjing, 210094, People’s Republic of China} \\
e-mail: {\bf{shaodong.wang@mail.mcgill.ca}}\\

\textsc{Ran Zhuo\\
	Mathematics and Science College,\\
	Shanghai Normal University\\
	Shanghai, 200240, People’s Republic of China} \\
AND\\
 {Department of Mathematical and Statistics,\\
	Huanghuai University\\
	Zhumadian, 463000, People’s Republic of China} \\
e-mail: {\bf{zhuoran1986@126.com}}

	\end{document}